\documentclass[12pt]{article}%

\usepackage[a4paper,margin=1in]{geometry}
\usepackage{amsmath}%
\usepackage{amsfonts}%
\usepackage{amssymb}%
\usepackage{graphicx}
\usepackage[utf8x]{inputenc}
\newtheorem{theorem}{Theorem}

\newtheorem{lemma}[theorem]{Lemma}

\newtheorem{remark}[theorem]{Remark}

\newenvironment{proof}[1][Proof]{\textbf{#1.} }{\ \rule{0.5em}{0.5em}}

\newcommand{\PP}{\mathbb{P}}

\DeclareMathOperator*{\essinf}{ess\,inf}

\begin{document}

\title{Sharp large deviation estimates  for Gaussian extrema}
\author{José M. Zapata\thanks{The author acknowledges the partial support of the Ministerio de Ciencia e Innovación
of Spain in the project PID2022-137396NB-I00, funded by MICIU/AEI/10.13039/501100011033
and by 'ERDF A way of making Europe'.}}
\date{}
\maketitle

\begin{abstract} 
We establish sharp large-deviation asymptotic estimates for the maximum order statistic of i.i.d.\ standard normal random variables on all Borel subsets of the positive real line. This result  yields more accurate tail approximations than the classical Gumbel limit.
  
\end{abstract}

\section{Main result}

Let $(X_n)_{n\in\mathbb{N}}$ be i.i.d.\ standard Gaussian random variables, and let
$X_{(n)}=\max_{1\le i\le n} X_i$ denote the maximum order statistic. A classical result in extreme value theory asserts that there exist sequences $a_n>0$ and $b_n\in\mathbb{R}$ such that
\[
Y_n = a_n\bigl(X_{(n)} - b_n\bigr)
\;\xrightarrow{d}\;
\Lambda,
\]
where $\Lambda(x)=e^{-e^x}$ is the standard Gumbel distribution.

The study of appropriate normalizations for maxima dates back to Fisher and Tippett \cite{{fisher-tippett}}, who first identified the possible limiting forms of extreme order statistics. Later, Gnedenko \cite{gnedenko} established necessary and sufficient conditions for convergence to the three classical extreme value distributions, with the Gaussian case belonging to the Gumbel domain of attraction. Cramér \cite{cramer} subsequently derived explicit asymptotic expressions for the normalizing constants in the Gaussian setting.

 Throughout this note, we will consider the standard choice of normalization sequences in the Gaussian case 
\[
a_n = \sqrt{2\log n},
\qquad
b_n = a_n - \frac{\log\!\bigl(\sqrt{2\pi}\,a_n\bigr)}{a_n}.
\]

While the Gumbel limit accurately captures  moderate fluctuations of
$X_{(n)}$ (for a fixed sample size $n$), it does not describe events lying  in the tail of the distribution,
whose probabilities decay much faster than predicted by the Gumbel. Such extreme tail probabilities are of central importance in risk analysis and actuarial sciences, see, e.g., \cite{embrechts,mcneil}.

In order to provide a more accurate estimates of the such a tail probabilities, in this note we adopt a large deviation theory approach. Large deviation theory concerns the asymptotic behavior of probabilities of rare events and characterizes their decay on an appropriate logarithmic scale, see e.g.~\cite{dembo}.

We investigate the rescaled normalization
\begin{equation}\label{eq:normalized}
Z_n
=
\frac{a_n\bigl(X_{(n)} - b_n\bigr)}{\log n}
=
\frac{X_{(n)} - b_n}{a_n/2}.
\end{equation}
 Our main result establishes a sharp large
deviation estimates of $(Z_n)_{n\in\mathbb{N}}$ for all Borel sets  on the right tail of the distribution, and reads as follows.

\begin{theorem}\label{thm:main}
Let $(Z_n)_{n\in\mathbb{N}}$ be defined by \eqref{eq:normalized}. Then, for every Borel set
$A\subset [0,\infty)$,
\begin{equation}\label{eq:LDP}
\lim_{n\to\infty}
\frac{1}{\log n}\,
\log \PP(Z_n\in A)
=
- \operatorname*{ess\,inf}_{x\in A}
\left( x + \frac{x^{2}}{4} \right).
\end{equation}
\end{theorem}

For a Borel set $A \subset [0,\infty)$ and large $n$, the theorem
yields the  approximation
\[
\mathbb{P}(Z_n \in A)
\;\approx\;
n^{-I_A},
\qquad
\text{with}
\qquad
I_A
=
\operatorname*{ess\,inf}_{x\in A}
\left( x + \frac{x^{2}}{4} \right).
\]
In Figure~\ref{fig:Zn}, we compare on a logarithmic
scale the true tail probability $\mathbb{P}(Z_n > x)$ (black curve), the standard Gumbel
approximation (blue curve), and the large-deviation approximation (orange curve). The figure shows that the large-deviation approximation significantly outperforms the standard Gumbel approximation in the tail regime. In addition, for sufficiently large
values of $x$, the true probability becomes so small that numerical evaluation
suffers from underflow and is rounded to zero; as a consequence, the black curve
(representing its logarithm) drops to $-\infty$. In contrast, the
large-deviation approximation  continues to provide an accurate description of the tail behavior.


\begin{figure}[t]
\centering
\includegraphics[width=0.75\textwidth]{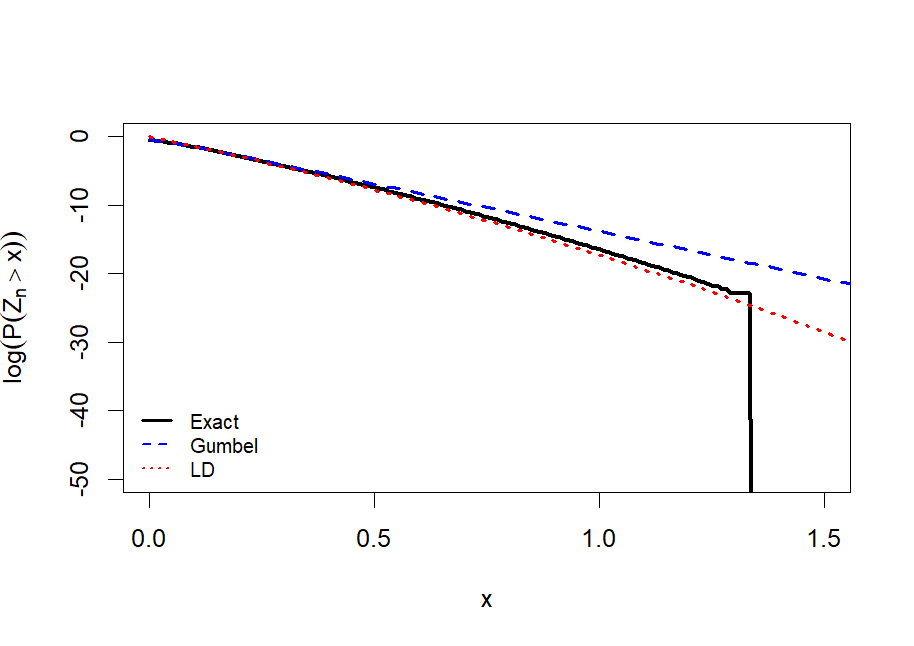}
\caption{Comparison of $\mathbb{P}(Z_n > x)$ on a logarithmic scale: true
probability (black), standard Gumbel approximation (blue), and large-deviation
approximation (orange). For large values of $x$, the true probability becomes
numerically indistinguishable from zero and its logarithm drops to $-\infty$.
The large-deviation approximation accurately captures the extreme tail
behavior, while the Gumbel approximation is only reliable for moderate
deviations.}
\label{fig:Zn}
\end{figure}


\section{Proof of the main result}

Let $(X_i)_{i\ge1}$ be i.i.d.\ standard Gaussian random variables and set $X_{(n)}=\max_{1\le i\le n} X_i$. Let $Z_n$ denote the normalized maximum defined in \eqref{eq:normalized}. In this section, we use the following notations
\[
t_n(x)=b_n+\tfrac{1}{2}a_n x,
\qquad
I(x)=x+\tfrac{x^2}{4}.
\]
Let $F_n$ denote the distribution function and $f_n$ the density function of $Z_n$. It is easy to verify that 
\begin{equation}\label{eq:distribution}
F_n(x)=\PP(Z_n\le x)=\Phi^n\bigl(t_n(x)\bigr),
\qquad
f_n(x)=\frac{n a_n}{2}\,\Phi^{\,n-1}\bigl(t_n(x)\bigr)\,\phi\bigl(t_n(x)\bigr),
\end{equation}
where $\Phi$ and $\phi$ denote the standard normal distribution and density.

In the following, we prove some preliminary results.
\begin{lemma}\label{lem:ineqNormal}
There exists $t_0$ such that
\begin{equation}\label{eq:normalI}
0 \le -\log \Phi(t) \le \frac{2\,\phi(t)}{t},
\qquad \mbox{ for all }t>t_0.
\end{equation}
Moreover, 
\begin{equation}\label{eq:normalII}
1-\Phi(t)^n \le n\,\frac{\phi(t)}{t},
\qquad\mbox{ for all }t>0\mbox{ and }n\in\mathbb{N}.
\end{equation}
\end{lemma}

\begin{proof}
Since $0<\Phi(t)\le1$, we trivially have $-\log\Phi(t)\ge0$.
By the standard upper bound in Mills' ratio (see e.g.~\cite[p.~45]{small}), we have
\[
1-\Phi(t)\le \frac{\phi(t)}{t}, \qquad t>0.
\]
 Using the elementary inequality
$-\log u\le 2(1-u)$ for $u<1$ close to $1$, and that  $\Phi(t)\to1$ as $t\to\infty$, we obtain
\[
-\log\Phi(t)
=
2(1-\Phi(t))
\le \frac{2\,\phi(t)}{t},
\]
for $t$ large enough, which proves \eqref{eq:normalI}.

As for \eqref{eq:normalII}, let $u\in[0,1]$. Then
\[
1-u^n=(1+u+\cdots+u^{n-1})(1-u)\le n(1-u).
\]
Applying this inequality with $u=\Phi(t)$ and using again
$1-\Phi(t)\le \phi(t)/t$ yields the result.
\end{proof}

\begin{lemma}\label{lem:limsup}
For every $x>0$,
\[
\limsup_{n\to\infty}\frac{1}{\log n}\log \PP(Z_n>x)\le -I(x).
\]
\end{lemma}

\begin{proof}
Fix $x>0$. 
Applying \eqref{eq:normalII} in  Lemma~\ref{lem:ineqNormal}, we obtain
\[
\frac{1}{\log n}\log\PP(Z_n>x)=\frac{1}{\log n}\log\left(1-\Phi\!\bigl(t_n(x)\bigr)^n \right)
\le
1+\frac{\log\phi\!\bigl(t_n(x)\bigr)}{\log n}
-\frac{\log t_n(x)}{\log n}.
\]
Since $\log t_n(x)/\log n\to0$ as $n\to\infty$ and
\[
\lim_{n\to\infty}\frac{\log\phi\!\bigl(t_n(x)\bigr)}{\log n}
=
-\left(1+\frac{x}{2}\right)^2,
\]
it follows that
\[
\limsup_{n\to\infty}\frac{1}{\log n}\log\PP(Z_n>x)
\le
1-\left(1+\frac{x}{2}\right)^2
=
-x-\frac{x^2}{4}
=
-I(x),
\]
which concludes the proof.
\end{proof}

\begin{lemma}\label{lem:density}
Let $Z_n$ be defined as above and let $f_n$ denote its density.
Then, for any fixed $M>0$,  
\[
\lim_{n\to\infty}\frac{1}{\log n}\log f_n(x)
=
-I(x),
\]
 uniformly on $x\in[0,M]$.
\end{lemma}

\begin{proof}
Since $a_n^2=2\log n$, we may write
\[
\frac{1}{\log n}\log f_n(x)
=
\frac{2}{a_n^2}\log\!\left(\frac{n a_n}{2}\right)
+
\frac{2(n-1)}{a_n^2}\log\Phi\!\bigl(t_n(x)\bigr)
+
\frac{2}{a_n^2}\log\phi\!\bigl(t_n(x)\bigr).
\]
Denote by $T_i(n,x)$, $i=1,2,3$, the i-th term in the sum above. We compute the limit of each of these terms and verify that the convergence is  uniform in  $[0,M]$.

First,
\[
\lim_{n\to\infty}T_1(n,x)=
1+\lim_{n\to\infty}\frac{\log a_n-\log 2}{\log n}=1.
\]
This convergence is obviously uniform since the
term does not depend on $x$.

Second, applying \eqref{eq:normalI} in Lemma \ref{lem:ineqNormal} taking into account that $t_n(x)\ge t_n(0)=b_n\to \infty$, we have
\begin{equation}\label{eq:T2}
|T_2(n,x)|\le-\frac{2(n-1)}{a_n^2}\log\Phi(b_n)\le 
\frac{4(n-1)}{a_n^2}\frac{\phi(b_n)}{b_n}
\end{equation}
for $n$ large enough and all $x\ge 0$. 
Inspection shows that  $n\phi(b_n)/b_n\to 1$ as $n\to \infty$. Then, the upper bound in \eqref{eq:T2} converges to $0$ as $n\to\infty$. Since it does not depend on $x$ the convergence is  uniform in  $x$.

Third, using that
\[
t_n(x)=a_n\left(1+\frac{x}{2}\right)-\frac{\log\left(\sqrt{2\pi}a_n\right)}{a_n},
\]
one gets
\begin{align*}
T_3(n,x)
&=
-\left(1+\frac{x}{2}\right)^2
+
\frac{2\left(1+\frac{x}{2}\right)\log(\sqrt{2\pi}a_n)}{a_n^2}  
-\frac{\log(2\pi)}{a_n^2}
-
\frac{\log^2(\sqrt{2\pi}a_n)}{a_n^4}.
\end{align*}
Then, for $x\in[0,M]$, we have
\[
\left|T_3(n,x)+\left(1+\frac{x}{2}\right)^2\right|
\le 
2\left(1+\frac{M}{2}\right)\frac{\log(\sqrt{2\pi}a_n)}{a_n^2}+\frac{\log(2\pi)}{a_n^2}
+
\frac{\log^2(\sqrt{2\pi}a_n)}{a_n^4}.
\]
The right hand side of the above inequality does not depend on $x$ and converges to $0$ as $n\to\infty$. Hence,
\[
T_3(n,x)
\longrightarrow
-\left(1+\frac{x}{2}\right)^2,\quad\mbox{ as }
n\to\infty
\]
uniformly for $x\in[0,M]$.

Combining the previous limits yields, for every $x\in[0,M]$,
\[
\lim_{n\to\infty}\frac{1}{\log n}\log f_n(x)
=
1-\left(1+\frac{x}{2}\right)^2
=
-x-\frac{x^2}{4}=-I(x).
\]
The convergence is  uniform in  the interval $[0,M]$. The proof is complete.
\end{proof}

We finally prove Theorem \ref{thm:main}.

\begin{proof}
Fix a Borel set $B\subset[0,\infty)$ with positive Lebesgue measure (otherwise the result is trivial). Set $I_B:=\essinf_B I$. 
Given $\varepsilon>0$, the set 
\[
B_\varepsilon:=\{x\in B\colon I(x)<I_B+\varepsilon\}
\]
has positive Lebesgue measure. Take $M>0$ large enough so that $[0,M]\cap B_\varepsilon$ has positive Lebesgue measure. 

By the uniform convergence in Lemma~\ref{lem:density}, there exists $n_0$ such that
\begin{equation}\label{eq:uniformLem}
n^{-I(x)-\varepsilon}\le f_n(x)\le n^{-I(x)+\varepsilon},
\qquad \mbox{ for all }x\in[0,M],\; n\ge n_0 .
\end{equation}

Next, to establish the existence of the limit in \eqref{eq:LDP} and identify its value, we derive matching upper/lower bounds for the corresponding limsup/liminf.

For the limit inferior, we have
\begin{align*}
\liminf_{n\to\infty}\frac{1}{\log n}\log\PP(Z_n\in B)&\ge \liminf_{n\to\infty}\frac{1}{\log n}\log\PP(Z_n\in B_\varepsilon\cap[0,M])\\
&= \liminf_{n\to\infty}\frac{1}{\log n}\log\int_{B_\varepsilon\cap[0,M]} f_n(x) dx\\
&\ge \liminf_{n\to\infty}\frac{1}{\log n}\log\int_{B_\varepsilon\cap[0,M]} n^{-I(x)-\varepsilon} dx\\
\\
&\ge \liminf_{n\to\infty}\frac{1}{\log n}\log\int_{B_\varepsilon\cap[0,M]} n^{-I_B-2\varepsilon} dx\\
&=-I_{B}-2\varepsilon.
\end{align*}

As for the limit superior, we obtain
\begin{align*}
\mathbb{P}(Z_n\in B)
&\le\mathbb{P}(Z_n\in B\cap[0,M])+\mathbb{P}(Z_n>M)\\
&\le
2\max\Big\{
\mathbb{P}(Z_n\in B\cap[0,M]),
\mathbb{P}(Z_n>M)
\Big\}.
\end{align*}
Hence,
\begin{align*}
\limsup_{n\to\infty}\frac{1}{\log n}\log \PP(Z_n\in B)
&\le
\max\Bigg\{
\limsup_{n\to\infty}\frac{1}{\log n}
\log \PP(Z_n\in B\cap[0,M]), \\[-1mm]
&\hspace{2.7cm}
\limsup_{n\to\infty}\frac{1}{\log n}
\log \PP(Z_n>M)
\Bigg\}.
\end{align*}

On the one hand, by Lemma \ref{lem:limsup}, we have
$$
\limsup_{n\to\infty}\frac{1}{\log n}\log \PP(Z_n>M)\le -I(M).
$$

On the other hand, we have
\begin{align*}
\limsup_{n\to\infty}\frac{1}{\log n}\log\PP(Z_n\in B)&
\le \limsup_{n\to\infty}\frac{1}{\log n}\log \int_{B\cap[0,M]} n^{-I_B+\varepsilon} dx\\
&=-I_B+\varepsilon.
\end{align*}

Putting all together, we have
\begin{align*}
-I_B-2\varepsilon&\le \liminf_{n\to\infty}\frac{1}{\log n}\log\PP(Z_n\in B)\\
 &\le \limsup_{n\to\infty}\frac{1}{\log n}\log\PP(Z_n\in B)\\
&\le \max\left\{-I_B+\varepsilon,-I(M)\right\}.
\end{align*}

Letting first $M\uparrow\infty$ and then $\varepsilon\downarrow0$, we obtain \eqref{eq:LDP}, 
which completes the proof.
\end{proof}

\begin{remark}
van der Hofstad and Honnappa \cite{van} derived sharp large deviation asymptotics for extrema of bivariate Gaussian random vectors. Their main result, however, is fundamentally different from ours (even in the degenerate case of perfect correlation), since  the normalization of the maxima differs from those considered here, leading to a different rate function $I(x)$. Moreover, while our result establishes a large deviation principle for all Borel sets of the positive real line, the results in \cite{van} are restricted to upper unbounded intervals. Finally, the comparison between the resulting large-deviation approximation and the classical Gumbel approximation is not addressed in \cite{van}.
\end{remark}

\end{document}